\documentclass[12pt]{amsart} 

\usepackage[dvips]{graphics} 
\usepackage{color,amssymb,amsbsy,amsmath,amsfonts,amssymb,
amscd,epsfig,times,graphics,enumerate}

\newtheorem{theo}{Theorem}
\newtheorem{prop}{Proposition}[section]
\newtheorem{lem}{Lemma}[section]

\theoremstyle{definition}
\newtheorem{defi}{Definition}[section]
\newtheorem{rem}{Remark}
\newtheorem{example}{Example}
 
\numberwithin{equation}{section}

\newcommand{\C}{\mathbb C} 
\newcommand{\R}{\mathbb R} 
\newcommand{\N}{\mathbb N} 
\newcommand{\B}{\mathbb B} 

\newcommand{\dis}{\displaystyle}

\begin{document}

\begin{abstract}
We define an analogue of the Baernstein star function for a meromorphic function $f$ in several complex variables. This function is subharmonic on the upper half-plane and encodes some of the main functionals attached to $f$.
We then characterize meromorphic functions admitting a harmonic star function. 
\end{abstract} 

\title[The star function for meromorphic functions of 
several complex variables]
{The star function for meromorphic functions of \\ several complex variables} 

\author{Faruk Abi-Khuzam, Florian Bertrand and Giuseppe Della Sala}

\thanks{}
\maketitle 

\section*{Introduction}

One aspect of the classical theory of meromorphic functions of finite order,
is the search for sharp asymptotic inequalities between certain functionals
associated with a given function $f$. Such functionals include, among
others, counting functions for $a$-values, the Nevanlinna characteristic or the  maximum modulus, denoted respectively by%
\[
N(r,a;f),T(r,f), M(r;f).
\]
There is a vast body of literature on those inequalities  notably for
functions of order less than one. We note in particular a unified approach
to some of those inequalties that has been presented by  J. Rossi  and A. Weitsmann in \cite{ro-we}  using  the theory of the Phragm\'en-Lindel\"of indicator along with the
Baernstein star function of $f$. The star function, denoted by $T^*(re^{i\theta
},f),$  was introduced by A. Baernstein \cite{ba1,ba2}, and used successfully by him in 
several problems beginning with the settlement of  Edrei's spread conjecture \cite{ed}. 
The crowning achievement in the use of the star function by A. Baernstein,
was the proof that the Koebe function is extremal for the $L^{p}$ norms of
all functions in the standard class $S$. A key ingredient in  Baernstein
proofs was the fact that while the star function of a typical meromorphic function $f$ is always
subharmonic in the upper half-plane, that of the extremal function is
harmonic.

The problems and techniques above have been considered and extended for subharmonic functions in $\R^{n}$ by many authors (see for instance \cite{ba-ta,ra-sh,ro-we,ha-ke}). In 
particular,  A. Baernstein and B. A. Taylor \cite{ba-ta} introduced an analogue of the star function in higher dimension. However, 
although such approach is rather natural for the study of subharmonic or $\delta$-subharmonic functions in $\R^n$, it does not seem that the star function introduced in  \cite{ba-ta} 
is well adapted to the distribution theory of entire, meromorphic or plurisubharmonic functions in several complex variables.   
 In this respect, the first author
had already suggested at least two possible definitions for a general star
function \cite{ab2} in several complex variables. 
In the present work, we follow one of those approaches and introduce the star function of a meromorphic function $F$ in $\C^n$ by averaging over the unit sphere the star functions $T^*(.,F_\zeta)$ of its \lq \lq slices\rq \rq \ $F_{\zeta }:\C \to \C$ defined by $F_{\zeta }(z)=F(z\zeta)$. Our first concern is to study the continuous dependence of $T^*(.,F_\zeta)$ on the parameter $\zeta$ (Theorem \ref{theocon}). In analogy with \cite{ab}, where the first author characterized all meromorphic functions admitting a harmonic star function in one variable (see also \cite{es-sh}), we provide a similar characterization in several complex variables  (Theorem \ref{theohar}). As might be expected, new elements enter the picture in the several variables case. In particular, it connects with the problem of determining a meromorphic function $F: \C^n \to \mathbb{C}$ from the knowledge of zero sets of its "slice" functions. We hope that our approach will allow to extend to several complex variables some of the known  inequalities in $\C$  and carry over a program similar to the one variable case. This will be  the focus of forthcoming work. 

The paper is organized as follows. In Section 1, we study the continuity on the unit sphere of $T^*(.,F_\zeta)$ with respect to 
$\zeta$ which allows us, in particular, to define an analogue of the Baernstein star function for meromorphic functions in several complex variables. In Section 2, we characterize meromorphic functions admitting a harmonic star function. 

\section{Star function for meromorphic function of several variables}
We denote by $\Delta_r=\{z\in \C  \ | \ |z|<r\}$ the disc in $\C$ centered atthe origin and of radius $r>0$. 
We denote the upper half-plane by $\mathbb{H}=\{z\in \C \ | \  \Im m z > 0\}$
and by $\mathbb{S}^{2n-1}$ the unit sphere in $\C^n$. 

\vspace{0.5cm}

Consider a meromorphic function $F:\C^n \to \C$ such that $F(0)=1$. Recall that $F$ can be written as $\displaystyle F=\frac{G}{H}$ where $G$ and $H$ are two coprime entire functions (see for instance Theorem 6.5.11 in \cite{kr}).  
 Define for $\zeta \in \mathbb{S}^{2n-1}$, the trace of $F$ on the complex line $\{z\zeta \ | \ z\in \C\}$, $ F_\zeta: \C \to \C$  by 
$$F_\zeta(z)=F(z\zeta).$$ 
For $t>0$ and $a \in \C\cup \{\infty\}$,  let $n(t, a; F_\zeta)$ be the number of $a$-points of $F_\zeta$ in the closed disc $\overline{\Delta_t}$.
For $a \in \{0,\infty\}$ and $r\geq 0$, the counting function of  $F_\zeta$ is defined by  
$$N(r, a; F_\zeta)=\int_0^r \frac{n(t, a; F_\zeta)}{t}dt.$$
Note that according to Jensen's formula, one has 
\begin{equation}\label{eqjen}
N(r, 0; F_\zeta)-N(r, \infty; F_\zeta)=\frac{1}{2\pi}\int_{-\pi}^{\pi} \log |F(re^{i\theta}\zeta)| d\theta.
\end{equation}
 
 For $\zeta \in  \mathbb{S}^{2n-1}$, we consider the  Baernstein star function associated to
$ F_\zeta: \C \to \C$ (see \cite{ba1,ba2}) 
$$T^*(re^{i\theta},F_\zeta)=\sup_E \ \frac{1}{2\pi}\int_E \log |F(re^{ix}\zeta)|dx+N(r,\infty;F_\zeta)$$   
where $re^{i\theta} \in \overline{\mathbb{H}}\setminus\{0\}$ and where the $\sup$ is taken over all sets $E \subset [-\pi,\pi]$ of Lebesgue measure $|E|=2\theta$.   We will write 
$$F_\zeta^*(re^{i\theta})=\sup_E \ \frac{1}{2\pi}\int_E \log |F(re^{ix}\zeta)|dx.$$
Note that 
$$T^*(r,F_\zeta)=N(r, \infty; F_\zeta)$$ 
and that Jensen's formula (\ref{eqjen}) implies 
$$T^*(-r,F_\zeta)=N(r, 0; F_\zeta).$$ 
The fundamental result of A. Baernstein states that $T^*(.,F_\zeta)$ is subharmonic on 
$\mathbb{H}$ and continuous on $\overline{\mathbb{H}}\setminus\{0\}$ \cite{ba1,ba2}; moreover, under the
assumption that  $F_{\zeta}(0)=1$,   $T^*(.,F_\zeta)$ extends continuously  on $\overline{\mathbb{H}}$. Our first main result is that for a 
fixed $re^{i\theta}$, $r>0$, $\theta \in [0,\pi)$ the map  $\zeta \mapsto T^*(re^{i\theta},F_\zeta)$
is continuous a.e. on the sphere $\mathbb{S}^{2n-1}$: 

\begin{theo}\label{theocon}
Let  $\displaystyle F=\frac{G}{H}:\C^n \to \C$ be a meromorphic function satisfying $F(0)=1$, where $G$ and $H$ are two coprime entire functions.
Define the following set
$$X=\{ \zeta \in \mathbb{S}^{2n-1}\ | \ G_\zeta^{-1}(0) \cap H_\zeta^{-1}(0)\neq \emptyset\}.$$
Then
\begin{enumerate}[i.]
\item The set $X$ has Lebesgue measure zero on $\mathbb{S}^{2n-1}$. 
\item For a fixed $re^{i\theta}$, $r>0$, $\theta \in [0,\pi)$, the function $\zeta \mapsto F_\zeta^*(re^{i\theta})$ is continuous on 
${\mathbb{S}^{2n-1}} \setminus X$.
\item For a fixed $r>0$ the function $\zeta \mapsto N(r, \infty; F_\zeta)$ is continuous on ${\mathbb{S}^{2n-1}} \setminus X$.  
\end{enumerate}
\end{theo}

In order to prove Theorem \ref{theocon} we first establish two lemmas which may be of independent interest.
Following A. Baernstein \cite{ba2}, we introduce   the level sets 
$$E(\zeta,t)=\{x \in [-\pi,\pi] \ | \ \log |F(re^{ix}\zeta)| >t\},$$
where $\zeta \in \mathbb{S}^{2n-1}$, $t \in \R$ and $r>0$. It follows from the proof of Proposition 1 in \cite{ba2} that for any $\zeta \in \mathbb{S}^{2n-1}$ there exists $t(\zeta) \in \R$ such that   
$$T^*(re^{i\theta},F_\zeta)=\frac{1}{2\pi}\int_{E(\zeta,t(\zeta))} \log |F(re^{ix}\zeta)|dx+N(r,\infty;F_\zeta)$$   
with $|E(\zeta,t(\zeta))|=2\theta$. Indeed, following A. Baernstein's notations in \cite{ba2}, in our case the distribution function 
$\lambda(t)=|E(\zeta,t)|$ of $F_\zeta$ 
is continuous since every level set of $F_\zeta$ has measure zero and one can take $E=A$. It follows that 
\begin{equation}\label{eqtstar1}
T^*(re^{i\theta},F_\zeta)=\frac{1}{2\pi}\int_{-\pi}^{\pi} \log^+ (\left|F(re^{ix}\zeta)|-t(\zeta)\right)dx+\frac{\theta t(\zeta)}{\pi}+N(r,\infty;F_\zeta),
\end{equation}
where $\log^+ (\left|F(re^{ix}\zeta)|-t(\zeta)\right)=\max \{\log |F(re^{ix}\zeta)|-t(\zeta),0\}$.

\begin{lem}\label{lemcont}
The function $\zeta \mapsto t(\zeta)$ is continuous on $\mathbb{S}^{2n-1}\setminus X$.  
\end{lem}
\begin{proof}
Fix $\zeta_0 \in \mathbb{S}^{2n-1}\setminus X$ and $\varepsilon>0$. Recall that $\displaystyle F=\frac{G}{H}$ where $G$ and $H$ are two coprime entire functions.   
Denote by  $p_j=re^{ix_j}\zeta_0 \in \C^n$ with $x_j \in [-\pi,\pi]$, $j=1,\cdots,N$,  the points such that $H(p_j)=0$. Note that
since $\zeta_0 \in \mathbb{S}^{2n-1}\setminus X$ then $G(p_j)\neq 0$.  There exists $\varepsilon'>0$ such that if $Z \in \cup_{j=1}^k \B(p_j,\varepsilon')$
then $\log |F(Z)|>t(\zeta_0)+1$. Here $\B(p_j,\varepsilon')$ denotes the open ball centered at $p_j$ and of radius $\varepsilon'$. We then chose $\delta>0$ such that if $|x-x_j|<\delta$ for some $j=1,\cdots,N$ and $\|\zeta-\zeta_0\|<\delta$ then
$re^{ix}\zeta \in \B(p_j,\varepsilon')$. Next we choose $t'$ large enough in such a way that if $x \in E(\zeta_0,t')$ then there exists 
$1\leq j\leq N$ such that $|x-x_j|<\delta$. Finally we consider a compact set $K\subset [-\pi,\pi]$  avoiding the singularities of $\log |F(re^{ix}\zeta_0)|$ and containing $E(\zeta_0,t(\zeta_0))\setminus E(\zeta_0,t')$. 
There exists $\delta'>0$ such that if $\|\zeta-\zeta_0\|<\delta'$ then 
$$\sup_{x\in K}|\log |F(re^{ix}\zeta)|-\log |F(re^{ix}\zeta_0)||<\varepsilon.$$ Let $x \in E(\zeta_0,t(\zeta_0))\setminus E(\zeta_0,t')$. 
Then $\log |F(re^{ix}\zeta_0)|>t(\zeta_0)$ and so 
\begin{equation}\label{eqineq1}
\log |F(re^{ix}\zeta)|>\log |F(re^{ix}\zeta_0)|- \varepsilon > t(\zeta_0)-\varepsilon
\end{equation}
 whenever $\|\zeta-\zeta_0\|<\delta'$. 
Now let $x \in E(\zeta_0,t')$. Then there is $1\leq j \leq N$ such that $|x-x_j|<\delta$. If $\|\zeta-\zeta_0\|<\delta$ then  $re^{ix}\zeta \in \B(p_j,\varepsilon')$ and therefore 
\begin{equation}\label{eqineq2}
\log |F(z)|>t(\zeta_0)+1.
\end{equation}
 It follows from (\ref{eqineq1}) and (\ref{eqineq2}) that
\begin{equation*}
E(\zeta_0,t(\zeta_0))\subset E(\zeta,t(\zeta_0)-\varepsilon).
\end{equation*}
whenever $\|\zeta-\zeta_0\|<\min\{\delta,\delta'\}$. Since $|E(\zeta_0,t(\zeta_0))|=|E(\zeta,t(\zeta))|=2\theta$, this implies $t(\zeta)\geq t(\zeta_0)-\varepsilon$. By symmetry we obtain $|t(\zeta)-t(\zeta_0)| \leq\varepsilon$ if $\|\zeta-\zeta_0\|<\min\{\delta,\delta'\}$. Therefore $\zeta \mapsto t(\zeta)$ is 
continuous on ${\mathbb{S}^{2n-1}}\setminus X$.
\end{proof}

\begin{lem}\label{lemjen} 
Let $H:\C^n \to \C$ be an entire function and let $r>0$. The function $\zeta \mapsto \int_{-\pi}^{\pi} \log |H(re^{ix}\zeta)| dx$ defined on 
$\mathbb{S}^{2n-1}$ is continuous.    
\end{lem}
\begin{proof} Let $\zeta_0 \in \mathbb{S}^{2n-1}$ and let $\varepsilon>0$. If $E \subset [-\pi,\pi]$ is a set then, following  \cite{ed-fu2} and using Lemma III in \cite{ed-fu1}, we have for any $\zeta \in \mathbb{S}^{2n-1}$ 
\begin{eqnarray*}
\frac{1}{2\pi}\int_{E} |\log |H(re^{ix}\zeta)||dx & \leq & m(r;H_{\zeta},E) + m\left(r;\frac{1}{H_{\zeta}},E\right)\\
\\
& \leq & c \left(T(2r,H_{\zeta})+T\left(2r,\frac{1}{H_{\zeta}}\right) \right) |E| \left(1+ \log^+\frac{1}{|E|} \right)    \\
\\
& \leq & 3c T(2r;H_{\zeta}) |E| \left(1+ \log^+\frac{1}{|E|} \right)    \\
\\
& \leq & 3c \log M(2r;H_{\zeta}) |E| \left(1+ \log^+\frac{1}{|E|} \right)    \\
\\
& \leq & c'(r)|E| \left(1+ \log^+\frac{1}{|E|} \right)=\Psi(r,E)\\
\end{eqnarray*}
where $c>0$ is a constant, $c'(r)>0$ is a constant depending only on $r$, and where 
$$m(r;H_{\zeta},E)=\frac{1}{2\pi}\int_E \log^+H_{\zeta}(re^{ix})dx.$$
Consider now $t_0<0$ with $-t_0$ large enough such that  $2\pi\Psi(r,[-\pi,\pi]\setminus E(\zeta_0,t_0))<\varepsilon$.  
There exists $\delta>0$ such that if $\|\zeta-\zeta_0\|<\delta$ then 
$$\sup_{x\in \overline{E(\zeta_0,t_0)}}\left|\log |H(re^{ix}\zeta)|-\log |H(re^{ix}\zeta_0)|\right|<\varepsilon.$$ 
Set 
$$I=\left| \int_{-\pi}^{\pi} \log |H(re^{ix}\zeta)|-\log |H(re^{ix}\zeta_0)| dx\right|.$$ For $\zeta \in \mathbb{S}^{2n-1}$ such that $\|\zeta-\zeta_0\|<\delta$, we have
\begin{eqnarray*}
I&\leq&\left| \int_{\overline{E(\zeta_0,t_0)}} \log |H(re^{ix}\zeta)|-\log |H(re^{ix}\zeta_0)| dx\right|\\
\\
& & +\left| \int_{[-\pi,\pi]\setminus E(\zeta_0,t_0)} \log |H(re^{ix}\zeta)|-\log |H(re^{ix}\zeta_0)| dx\right|\\
\\
& \leq &  \int_{\overline{E(\zeta_0,t_0)}} \left|\log |H(re^{ix}\zeta)|-\log |H(re^{ix}\zeta_0)|\right| dx+ \int_{[-\pi,\pi]\setminus E(\zeta_0,t_0)} \left|\log |H(re^{ix}\zeta)|\right| dx\\
\\
& & +  \int_{[-\pi,\pi]\setminus E(\zeta_0,t_0)} \left|\log |H(re^{ix}\zeta_0)|\right| dx\\
\\
& \leq & \varepsilon + 4\pi \Psi\left(r,[-\pi,\pi]\setminus E(\zeta_0,t_0)\right).\\
\end{eqnarray*}
This proves the continuity of $\zeta \mapsto \int_{-\pi}^{\pi} \log |H(re^{ix}\zeta)| dx$ on the sphere $\mathbb{S}^{2n-1}$. 
\end{proof}
We now prove Theorem \ref{theocon}.
\begin{proof}[Proof of Theorem \ref{theocon}]
We prove $i.$ 
 Let $\mathcal Z\subset \mathbb C^n$ be the \emph{indeterminacy set} of $F$, that is 
 \begin{equation}\label{eqind}
 \mathcal Z=\{Z\in \mathbb C^n \ | \  G(Z)=H(Z)=0\}.
 \end{equation}
By the assumptions on $F$, $\mathcal Z$ is a complex analytic subvariety of $\mathbb C^n$ of complex dimension at most $n-2$. Let $\tau:\mathbb C^n\to\mathbb C\mathbb P^{n-1}$ be the projection of $\mathbb C^n$ onto the projective space $\mathbb C\mathbb P^{n-1}$. Note that the restriction $\tau_{|\mathbb S^{2n-1}}$ is a constant rank map $\mathbb S^{2n-1}\to\mathbb C\mathbb P^{n-1}$; it is indeed a fibration - the \emph{Hopf fibration} - with fiber $\mathbb S^1$. Therefore for any subset $K\subset \mathbb C\mathbb P^{n-1}$ we have that the ($2n-2$-dimensional) Lebesgue measure of $K$ vanishes if and only if  the ($2n-1$-dimensional) Lebesgue measure of the inverse image $\tau_{|\mathbb S^{2n-1}}^{-1} (K)\subset \mathbb S^{2n-1}$ is zero. Since by definition $X=\tau_{|\mathbb S^{2n-1}}^{-1}(\tau(\mathcal Z))$, to prove $i.$ it is enough to show that $\tau(\mathcal Z)\subset \mathbb C\mathbb P^{n-1}$ has measure $0$.

Since $\mathcal Z$ is a ($n-2$)-dimensional complex subvariety of $\mathbb C^n$, there exists a countable collection $\{\mathcal Z_j\}_{j\in \mathbb N}$ of locally closed, non-singular complex submanifolds of $\mathbb C^n$, each one of dimension at most $n-2$, such that $\mathcal Z=\cup_{j\in \mathbb N} \mathcal Z_j$. Fixed $j\in \mathbb N$, consider the restriction $\tau_{|\mathcal Z_j}:\mathcal Z_j\to\mathbb C\mathbb P^{n-1}$. The map $\tau_{|\mathcal Z_j}$ is smooth (and in fact analytic), and its rank at any point $p$  of $\mathcal Z_j$ is less than $n-1$ since $\dim_{\mathbb C} \mathcal Z_j\leq n-2$, hence all $p\in \mathcal Z_j$ are critical points of $\tau_{|\mathcal Z_j}$. It follows by Sard's theorem that $\tau(\mathcal Z_j)$ has measure zero. Since $\tau(\mathcal Z)\subset \cup_{j\in \mathbb N} \tau(\mathcal Z_j)$ we conclude that $\tau(\mathcal Z)$ has measure zero.

\vspace{0.5cm}

We now prove $ii.$ 
We fix $re^{i\theta}$ with $r>0$ and $\theta \in [0,\pi)$.
According to Equation (\ref{eqtstar1}) and Lemma \ref{lemcont} we only need to show that the function  
$$\zeta \mapsto \int_{-\pi}^{\pi} \log^+ \left(|F(re^{ix}\zeta)|-t(\zeta)\right)dx$$
defined on $\mathbb{S}^{2n}\setminus X$ is continuous. Let $\zeta_0 \in \mathbb{S}^{2n}\setminus X$ and let $\varepsilon>0$.
Set 
$$J=\left|\int_{-\pi}^{\pi} \log^+ \left(|F(re^{ix}\zeta)|-t(\zeta)\right)- \log^+\left(|F(re^{ix}\zeta_0)|-t(\zeta_0)\right)dx \right|.$$ For $\zeta \in  \mathbb{S}^{2n}\setminus X$ such that 
$\|\zeta-\zeta_0\|<\delta$ we have
\begin{eqnarray*}
J&\leq &\int_{-\pi}^{\pi} \left|\left(\log |F(re^{ix}\zeta)|-t(\zeta)\right)- \left(\log |F(re^{ix}\zeta_0)|-t(\zeta_0)\right)\right|dx \\
\\
& \leq & \int_{-\pi}^{\pi} \left|\log |G(re^{ix}\zeta)|- \log |G(re^{ix}\zeta_0)|\right|dx\\
\\
& &
+ \int_{-\pi}^{\pi} \left|\log |H(re^{ix}\zeta)|- \log |H(re^{ix}\zeta_0)|\right|dx+
\int_{-\pi}^{\pi} \left|t(\zeta)-t(\zeta_0)\right|dx\\
\end{eqnarray*}
The statement $ii.$ now follows from Lemma \ref{lemcont} and Lemma \ref{lemjen}.

\vspace{0.5cm}

Finally  $iii.$ follows directly from Lemma \ref{lemjen} since  when $\zeta \in \mathbb{S}^{2n}\setminus X$ we have
$$N(r, \infty; F_{\zeta})=N(r, 0; H_{\zeta})$$ 
and by Jensen formula (\ref{eqjen})
$$N(r, 0; H_{\zeta})=\frac{1}{2\pi}\int_{-\pi}^{\pi}\log |H(re^{ix}\zeta)|\ dx.$$ 
\end{proof}
Notice that in case the set $X$ is empty, Theorem \ref{theocon} implies that,  for a fixed $re^{i\theta}$, $r>0$, $\theta \in [0,\pi)$, 
 the functions  $\zeta \mapsto F_\zeta^*(re^{i\theta})$ and  $\zeta \mapsto N(r, \infty; F_\zeta)$
 are continuous on ${\mathbb{S}^{2n-1}}$. This is in particular the case when $F:\C^n \to \C$ is entire, or meromorphic without zeros.  
However, note that in general the function $\zeta \mapsto N(r, \infty; F_\zeta)$ may not be continuous on ${\mathbb{S}^{2n-1}}$:
\begin{example}
Consider the meromorphic function on $\C^2$ defined by  
$$F(z_1,z_2)=\frac{z_1-1}{z_2-1}.$$
Then for any $r>0$, we have $N(r, \infty; F_{\zeta_{0}})=0$ for $\zeta_0=\left(\frac{1}{\sqrt{2}},\frac{1}{\sqrt{2}}\right)$. Now for 
$\zeta_k=(r_k,s_k) \in {\mathbb{S}^{2n-1}}$ converging to $\zeta_0$ we have       
$$N(2, \infty; F_{\zeta_k})=\int_0^2 \frac{n(t, \infty; F_{\zeta_k})}{t}dt=\log 2 + \log s_k$$
since $n(t, \infty; F_{\zeta_k})$ equals $0$ for $0<t<1/s_k$ and $1$ for $t\geq 1/s_k$.   
It is interesting to notice that the function $\zeta \mapsto F_\zeta^*(2e^{i\theta})$ is continuous at $\zeta_0$. Indeed it can checked that if $r_k>s_k>0$ we have  
$$F_{\zeta_k}^*(2e^{i\theta})=\frac{1}{2\pi}\int_{-\theta}^{\theta}\log\left|\frac{r_ke^{ix}-1}{s_ke^{ix}-1}\right|dx$$
and if $s_k>r_k>0$ then  
$$F_{\zeta_k}^*(2e^{i\theta})=\frac{1}{2\pi}\int_{\pi-\theta}^{\pi+\theta}\log\left|\frac{r_ke^{ix}-1}{s_ke^{ix}-1}\right|dx.$$
In both cases $F_{\zeta_k}^*(2e^{i\theta})\to 1$ as $\zeta_k \to \zeta_0$. However note that the set $E(\zeta)$ realizing the 
supremum in $F_\zeta^*(2e^{i\theta})$ does not depend continuously on $\zeta$.     
\end{example}

For a fixed $re^{i\theta}$, $r>0$, $\theta \in [0,\pi)$, 
since $\zeta \mapsto T^*(re^{i\theta},F_\zeta)$ is bounded, Theorem \ref{theocon} shows in particular integrability of  $T^*(re^{i\theta},F_\zeta)$
on the unit sphere and therefore allows us to define an analogue of the Baernstein star function associated to a meromorphic function $F$ of several complex variables.   
\begin{defi}
Let  $F:\C^n \to \C$ be a meromorphic function satisfying $F(0)=1$. The star function of $F$ is defined by 
\begin{equation}\label{eqtstar}
T^*(re^{i\theta},F)=\frac{1}{\sigma_{2n-1}}\int_{\mathbb{S}^{2n-1}} T^*(re^{i\theta},F_\zeta) d\sigma(\zeta),
\end{equation}
where $re^{i\theta} \in \overline{\mathbb{H}}\setminus\{0\}$ and $d\sigma$ denotes the Lebesgue surface area measure of the 
$\mathbb{S}^{2n-1}$ and $\sigma_{2n-1}$ its area.
\end{defi}
\begin{rem}
In order to show the integrability of the counting function $N$, it is not strictly necessary to show its continuity. Indeed in case $H:\C^n\to \C$ is an entire function then according to Jensen's formula, for a fixed $r>0$, 
the positive function $\zeta \mapsto N(r,0,H_\zeta)$ is plurisubharmonic (see Proposition 
I.14 in \cite{le-gr}) and therefore $L^1$ on the unit sphere $\mathbb{S}^{2n-1}\subset \C^n$. Now, with respect to the notations of 
Theorem \ref{theocon}, we have $N(r,\infty,F_\zeta)=N(r,0,H_\zeta)$ for $\zeta \in \mathbb{S}^{2n-1}\setminus X$ and so   
$\zeta \mapsto N(r,\infty,F_\zeta)$ is  $L^1$ on $\mathbb{S}^{2n-1}$.
\end{rem}

For $a \in \{0,\infty\}$ we set
\begin{equation*}
N(r, a; F)=\frac{1}{\sigma_{2n-1}}\int_{\mathbb{S}^{2n-1}} N(r, a; F_\zeta) d\sigma(\zeta).
\end{equation*}
The function $N(r, a; F)$ can also be expressed as 
 \begin{equation}\label{eqcount2}
N(r, a; F)=\int_0^r \frac{n(t, a; F)}{t}dt.
\end{equation}
where $\displaystyle n(t, a; F)=\frac{1}{\sigma_{2n-1}}\int_{\mathbb{S}^{2n-1}}n(t, a; F_\zeta)d\sigma(\zeta)$ is, for $a=0$, 
 the Lelong number of the zero set of $F$ (see \cite{le-gr} for instance). 
Notice that since $F(0)=1$
$$T^*(r,F)=N(r, \infty; F),$$ and 
$$T^*(-r,F)=N(r, 0; F).$$ 

In the next proposition, we extend to several variables the main property of the star function (\ref{eqtstar}):
\begin{prop}\label{propsub}
Let  $F:\C^n \to \C$ be a meromorphic function satisfying $F(0)=1$. Then the function $T^*(.,F)$ is subharmonic on 
$\mathbb{H}$.
\end{prop}
\begin{proof}
Let $z_0 \in \mathbb{H}$ and let $r>0$ such that the closed disc centered at $z_0$  and radius $r$ is included in $\mathbb{H}$. We have 
\begin{eqnarray*}
\frac{1}{2\pi}\int_{-\pi}^{\pi} T^*(z_0+re^{i\theta},F)d\theta & = & \frac{1}{2\pi}\int_{-\pi}^{\pi} \frac{1}{\sigma_{2n-1}}\int_{\mathbb{S}^{2n-1}} T^*(z_0+re^{i\theta},F_\zeta) d\sigma(\zeta) d\theta \\
\\
 & = &\frac{1}{\sigma_{2n-1}}\int_{\mathbb{S}^{2n-1}} \frac{1}{2\pi}\int_{-\pi}^{\pi}  T^*(z_0+re^{i\theta},F_\zeta) d\theta d\sigma(\zeta)  \\
 \\
  & \geq &\frac{1}{\sigma_{2n-1}}\int_{\mathbb{S}^{2n-1}} T^*(z_0,F_\zeta) d\sigma(\zeta)  \\
  \\
& =&   T^*(z_0,F)
\end{eqnarray*}
where the second equality follows from Theorem \ref{theocon} and the  inequality from the fact that the usual  Baernstein star function is
subharmonic. Therefore $T^*(.,F)$ is subharmonic  on $\mathbb{H}$.
\end{proof}
It is important to notice that the proof shows that $T^*(.,F)$ is harmonic  on $\mathbb{H}$ if and only if  $T^*(.,F_\zeta)$ is harmonic for a.e. $\zeta \in \mathbb{S}^{2n-1}$. This fact will be used in the proof of Theorem \ref{theohar}.

Now, Theorem \ref{theocon} and the continuity of the usual  Baernstein star function  implies 
directly that:
\begin{prop}Let  $F:\C^n \to \C$ be a meromorphic function satisfying $F(0)=1$. Then the function $T^*(.,F)$ is continuous on $\overline{\mathbb{H}}$.
\end{prop}
 Note that the continuity on $\{re^{i\theta} \in \C \ | \ \theta=\pi\}$ and on $\{re^{i\theta} \in \C \ | \ \theta=0\}$ follows from $(\ref{eqcount2}).$

 \section{Entire functions of several variables with harmonic star function}
In the case of  complex dimension one, as pointed out by A. Baernstein in \cite{ba1}, meromorphic functions of the kind
$$ f(z)=\prod_m \left(1+\frac{z}{r_m}\right)/\prod_m \left(1-\frac{z}{s_m}\right)$$  
where $r_m, s_m>0$ for all integer $m>0$ with $\sum_m \frac{1}{r_{m}}+\sum_m \frac{1}{s_{m}}<\infty$, admit a harmonic star function. In \cite{ab}, the first author characterized all  meromorphic functions with a  harmonic star function; see also the work of M. Ess\'en and D. F. Shea in \cite{es-sh} for the case of meromorphic functions of zero genus. More precisely, it was proved in  \cite{ab} that if $f: \C \to \C$ is a meromorphic function satisfying $f(0)=1$ and such that its star function is harmonic then $f$ can be written $f(z)=P(e^{i\theta}z)$ with
\begin{equation}\label{eqp}
\displaystyle P(z)=e^{\gamma z}\prod_m \left(1+\frac{z}{r_m}\right)/\prod_m \left(1-\frac{z}{s_m}\right),
\end{equation}
where $\theta \in \R$, $\gamma \geq 0$ and $r_m, s_m>0$ for all $m$ with $\sum_m \frac{1}{r_{m}}+\sum_m \frac{1}{s_{m}}<\infty$. From a geometric 
viewpoint, if the star function of $f$ is harmonic then the zeros of $f$ are distributed on one ray and its poles on 
the opposite ray. 

In this section, we characterise meromorphic functions $F: \C^n\to \C$ of several complex variables admitting a harmonic star function. 

\begin{theo}\label{theohar}
Let $F$ be a meromorphic function on $\C^n$ with $F(0)=1$. The star function $T^*(.,F)$ is harmonic on $\mathbb{H}$ if 
and only if there exist a meromorphic function $P:\C \to \C$ of the form (\ref{eqp}) and a vector $\eta=(\eta_1,\ldots,\eta_n)\in \mathbb C^n$ such that $F(Z)=P(Z\cdot\eta)$ for all $Z \in \mathbb C^n$, where we denote $Z\cdot \eta=z_1\eta_1+\ldots + z_n\eta_n$. In particular 
if $T^*(.,F)$ is harmonic on $\mathbb{H}$, then the indeterminacy set of $F$ as defined in (\ref{eqind}) is empty and 
for all $\zeta \in \mathbb{S}^{2n-1}$ the star function $T^*(.,F_{\zeta})$ is harmonic.   
\end{theo}
\begin{rem}
When $F$ is nonconstant the function $P$ is given by a (rescaled) restriction of $F$ to the complex line 
$\{z\partial F(0)\ | \ z\in \C\}$, where $\partial F(0)=\left(\frac{\partial F}{\partial z_1}(0), \cdots,\frac{\partial F}{\partial z_n}(0)\right)$. 
\end{rem}

We first establish the two following lemmas
\begin{lem}\label{lemcoef}
Let $f:\C \to \C$ be a meromorphic function of the form
\begin{equation}\label{eqformf}
f(z)=e^{\gamma e^{i\theta}z} \prod\limits_{m}\left( 1+\frac{e^{i\theta}z}{r_{m}}%
\right) /\prod_{m}\left( 1-\frac{e^{i\theta }z}{s_{m}}%
\right) 
\end{equation}
with $\theta \in \R$, $\gamma \geq 0$, $r_m, s_m>0$ for all $m$ and  
\[
\sum_m \frac{1}{r_{m}}+\sum_m \frac{1}{s_{m}}<\infty. 
\]%
Assume furthermore that $f$ has at least one zero or pole. Then%
\begin{equation}\label{eqf}
f^{(k)}(0)=d_{k}e^{ik\theta}=\frac{d_k}{d_1^k} \cdot (f'(0))^k, 
\end{equation}
where $d_{k}$ is real for all $k\geq 0$.
\end{lem}

\begin{proof}
For $z$ small enough we have
\begin{eqnarray*}
\log f(z)& =& \gamma e^{i\theta}z+\sum_{m}\sum_{k=1}^{\infty }\frac{(-1)^{k+1}}{k}\left( \frac{%
e^{i\theta }z}{r_{m}}\right) ^{k}+\sum_{m}\sum_{k=1}^{\infty }\frac{1}{k}%
\left( \frac{e^{i\theta }z}{s_{m}}\right) ^{k}\\ 
\\
&=& \gamma e^{i\theta}z+\sum_{k=1}^{\infty }\left( \sum_{m}\frac{(-1)^{k+1}}{r_{m}^{k}}+\sum_{m}%
\frac{1}{s_{m}^{k}}\right) \frac{e^{ik\theta }}{k}z^{k}\\
\\
&=&\sum_{k=1}^{\infty }c(k)\frac{e^{ik\theta }}{k}z^{k},\\
\end{eqnarray*}
where 
\[
\displaystyle c(1)=\gamma +\sum_{m}\frac{1}{r_{m}}+\sum_{m}\frac{1}{s_{m}}
\]%
and for $k\geq 2$ 
\[
\displaystyle c(k)=\sum_{m}\frac{(-1)^{k+1}}{r_{m}^{k}}+\sum_{m}\frac{1}{s_{m}^{k}}. 
\]%
This tells us that for $k\geq 1$ 
$$
D^{k}\log f(z)|_{z=0}=\frac{k!}{k}c(k)e^{ik\theta }$$
and so for $k\geq 0$
$$D^{k}\frac{f^{\prime }}{f}(0)=k!c(k+1)e^{i(k+1)\theta}. $$
Since $f(0)=1$ and $\displaystyle c(1)=\gamma+\sum_{m}\frac{1}{r_{m}}+\sum_{m}\frac{1}{s_{m}}>0$, and
since we have at least one zero or pole, we have 
$$f^{\prime}(0)=c(1)e^{i\theta }\neq 0.$$ We set $d_{0}=1$, $d_{1}=c(1) \in \R$. We now proceed by
induction. Having $\displaystyle f^{(k)}(0)=d_{k}e^{ik\theta }$ with $d_{k} \in \R$ for $%
0\leq k\leq m$, we have 
\begin{eqnarray*}
f^{(m+1)}(0)&=&\sum_{k=0}^{m}{m \choose k}D^{k}\frac{f^{\prime }}{f}(0)D^{m-k}f(0)\\
\\
&=&\sum_{k=0}^{m}{m \choose k}k!c(k+1)e^{i(k+1)\theta
}d_{m-k}e^{i(m-k)\theta }\\
\\
&=&\left( \sum_{k=0}^{m}{m \choose k}k!c(k+1)d_{m-k}\right) e^{i(m+1)\theta },
\end{eqnarray*}
which proves the first equality in (\ref{eqf}). The second equality follows directly. 
\end{proof}

\begin{lem}\label{lemexp}
Let $F$ be a meromorphic function on $\C^n$ with $F(0)=1$. Assume that its star function $T^*(.,F)$ is harmonic on $\mathbb{H}$. For any integer $k>0$  let $P_k$ be the polynomial 
giving the $k$-homogeneous part of the Taylor expansion of $F$ at the point $0$. 
Then there exists a sequence $\{c_k\}_{k\geq 2}$ of real numbers such that 
$$P_k=c_k (P_1)^k$$ for all $k\geq 2.$ 
\end{lem}

\begin{proof}
Since $T^*(.,F)$ is harmonic on $\mathbb{H}$, then by the definition of $T^*(.,F)$ and the proof of Proposition \ref{propsub}, for a.e. $\zeta \in \mathbb{S}^{2n-1}$, $T^*(.,F_\zeta)$ 
is harmonic on $\mathbb{H}$. Thus by Theorem 1 in \cite{ab}, for a.e. $\zeta \in \mathbb{S}^{2n-1}$, $F_\zeta$ has the form (\ref{eqformf}) in Lemma \ref{lemcoef}. So we have
$$F_\zeta^{(k)}(0)=d_{k}(\zeta)e^{ik\theta _{\zeta }},$$
with $d_{k}(\zeta) \in \R$ for all $k\geq 0$ for a.e. $\zeta \in \mathbb{S}^{2n-1}$, and thus for all $\zeta \in \mathbb{S}^{2n-1}$ by continuity of $d_k(\zeta)$. 
For $z\in \C$ and $\alpha=(\alpha_1,\cdots,\alpha_n) \in \C^n\setminus\{0\}$,  define 
$$F_{\alpha}(z)=F(\alpha_1 z,\cdots,\alpha_n z).$$ Since 
$\displaystyle F_{\alpha}(z)=F_{\frac{\alpha}{\|\alpha\|}}(\|\alpha\|z)$, the function $F_{\alpha}$ has the form 
(\ref{eqformf}) of Lemma \ref{lemcoef} and so $F_{\alpha}^{(k)}(0)$ is a real
multiple of $(F_{\alpha}^{\prime }(0))^{k}$. In particular, note that if  $\alpha \mapsto F_{\alpha}^{\prime }(0)$ is identically equal to  zero then $F$ must be identically equal to $1$.  
The homogeneous polynomial $P_k$ is given by  
$$P_k(Z)=\sum_{|J|=k}\frac{\partial_J F(0)}{J!} Z^J
,$$
where, for a multiindex $J=(j_1,\cdots, j_n) \in \N^n$, we write 
$J!=j_1!\cdots,j_n!$, $Z^J=z_1^{j_1}z_2^{j_2}\cdots z_n^{j_n}$ and  
$\displaystyle \partial_{J}F(0)=\frac{\partial^{j_1}}{\partial z_1^{j_1}}\cdots \frac{\partial^{j_n}}{\partial z_1^{j_n}} F(0).$
Therefore, we have 
\begin{eqnarray*}
F_\alpha(z)&=&\sum_k P_k(z\alpha)\\
\\
&=&\sum_k \sum_{|J|=k}\frac{\partial_J F(0)}{J!} (z\alpha)^J\\
\\
&=&\sum_k \left(\sum_{|J|=k}\frac{\partial_J F(0)}{J!} \alpha^J\right)z^k.\\
\end{eqnarray*}
It follows that
$$F_{\alpha}^{(k)}(0)=\displaystyle \sum_{|J|=k} \frac{k!}{J!}\partial_{J}F(0)\alpha^J=k!P_k(\alpha).$$
Since the function $\displaystyle \alpha \mapsto \frac{P_k(\alpha)}{(P_1(\alpha))^k} = \frac{F_{\alpha}^{(k)}(0)}{k!(F_{\alpha}^{\prime }(0))^{k}}$ is meromorphic and real valued, it is constant. This concludes the proof of Lemma \ref{lemexp}.
\end{proof}
\begin{rem}\label{remlem}
It follows from  the proof of Lemma \ref{lemexp} that it is enough to assume the following: there exists a subset $B \subset \mathbb{S}^{2n-1}$ of positive measure such that for all $\zeta \in B$, $T^*(.,F_\zeta)$ is harmonic on $\mathbb{H}$. Indeed this implies that the meromorphic function $\displaystyle \alpha \mapsto \frac{P_k(\alpha)}{(P_1(\alpha))^k}$ is  real valued for all $\alpha \in \C^n$ such that $\frac{\alpha}{\|\alpha\|} \in B$, which is enough to show that it is constant. 
\end{rem}
We are now able to prove Theorem \ref{theohar}.
\begin{proof}[Proof of Theorem \ref{theohar}]
Let $F$ be a meromorphic function on $\mathbb C^n$ sastifying $F(0)=1$. For any $k\in \mathbb N$ let $P_k$ be the polynomial 
giving the $k$-homogeneous part of the Taylor expansion of $F$ at the point $0$; in particular, the Taylor 
series of $F$ is given by $\sum_{k\in \mathbb N} P_k$.  If furthermore $F$ is such that $T^*(.,F)$ is harmonic on 
$\mathbb{H}$, then by Lemma \ref{lemexp}, there exists a sequence $\dis \{c_k\}_{k\geq 2}$ of real numbers such that 
\begin{equation}\label{eqpol}
P_k=c_k (P_1)^k
\end{equation}
 for all $k\geq 2$. 
For $j=1,\ldots,n$ define $\displaystyle \eta_j=\frac{\partial F}{\partial z_j}(0)$, so that $\dis P_1(Z)=\sum_{j=1}^n \eta_j z_j,$ and set $\eta=(\eta_1,\ldots,\eta_n)$. 

Suppose first that $\eta=0$. Then $P_1 = 0$, and by (\ref{eqpol}) we also have $P_k=0$ for all $k\geq 2$. It follows that $F$ is identically equal to $1$, and defining $P(z)\equiv 1$ we get $P(Z\cdot \eta)=P(0)=1=F(Z)$ for all $Z\in \mathbb C^n$.

Suppose then that $\eta\neq 0$. Without loss of generality, up to a permutation of the coordinates we can assume that $\eta_1\neq 0$. We define a meromorphic function $P$ as 
$$P(z)=F\left(\frac{z}{\eta_1},0,\ldots,0\right).$$ We need to show that $F(Z)=P(Z\cdot \eta)$ for all $Z\in \mathbb C^n$. We check this identity by verifying that the Taylor expansions of the two functions about $0$ coincide. On the one hand, 
using Lemma \ref{lemexp}, we have the polynomial identity
\[\sum_{|J|=k}\frac{\partial_J F}{J!}(0) Z^J=c_k\left(\sum_{j=1}^n \eta_jz_j\right)^k=c_k\sum_{|J|=k}\frac{k!}{J!}\eta^JZ^J\]
which implies that 
$$\partial_J F(0)=k!c_k\eta^J.$$ 
Using now the chain rule iteratively and denoting $G(Z)=P(Z\cdot \eta)$, for any multiindex $J=(j_1,\ldots,j_n)\in \mathbb N^n$ with $|J|=k\geq 2$ we have
\[\dis \partial_JG(0)=P^{(k)}(0)\eta_1^{j_1}\eta_2^{j_2}\cdots\eta_n^{j_n}=P^{(k)}(0)\eta^J.\]
To conclude the proof we just need to check that  $P^{(k)}(0)=k!c_k$.
First, we observe that $P'(0)=1$ and 
\[P^{(k)}(0)=\frac{1}{\eta_1^k}\frac{\partial^k F}{\partial z_1^k}(0)\] 
for all integer $k\geq 0$ by the chain rule. Using Lemma \ref{lemexp} again, we can write
\[\frac{1}{k!{\eta_1^k}}\frac{\partial^k F}{\partial z_1^k}(0)z^k= P_k\left(\frac{z}{\eta_1},0,\ldots,0\right)=c_k \left(P_1\left(\frac{z}{\eta_1},0,\ldots,0\right)\right)^k=c_k z^k.\]
Hence $$P^{(k)}(0)=k!c_k$$ for all integers $k\geq 2$. It follows that $\partial_J F(0)=\partial_J G(0)$ for all $J\in \mathbb N^n$, hence $F(Z)\equiv G(Z)=P(Z\cdot\eta)$.

Finally, if $\zeta\cdot \eta\neq 0$ and   $T^*(.,F_{\zeta})$ is harmonic then $T^*(.,P(z\zeta\cdot \eta))$, and hence $T^*(.,P)$, is 
harmonic. It follows by \cite{ab} that $P$ has the required form.     
\end{proof}

\begin{rem} According to the proof of Theorem \ref{theohar} and Remark \ref{remlem}, if we replace the harmonicity of $T^*(.,F)$ by the weaker statement that 
 there exists a subset $B \subset \mathbb{S}^{2n-1}$ of positive measure such that for all $\zeta \in B$, $T^*(.,F_\zeta)$ is harmonic on $\mathbb{H}$, the theorem still holds.
\end{rem}

\begin{example} It is not enough to assume that $T^*(.,F_\zeta)$ is harmonic on $\mathbb{H}$ for finitely many $\zeta \in \mathbb{S}^{2n-1}$. For instance in
 $\C^2$ consider $\zeta_1=(\alpha_1,\beta_1), \cdots,\zeta_N=(\alpha_1,\beta_1) \in \mathbb{S}^3$ and define, for $j=1,\cdots,N$, the linear function 
 $L_j:\C^2 \to \C$ $L_j(z_1,z_2)=\beta_jz_1-\alpha_jz_2.$ Then the function $F:\C^2 \to \C$ defined by 
$$F(z_1,z_2)=\prod_{j=1}^N L_j(z_1,z_2)+1,$$
is meromorphic with $F(0)=1$. Moreover $F_{\zeta_j}\equiv 1$ hence $T^*(.,F_{\zeta_j})$ is harmonic for $j=1,\cdots,N$ but $F$ is not of the form of Theorem 
\ref{theohar}.   
\end{example}

\begin{example} Note that if $T^*(.,F_\zeta)$ is harmonic on $\mathbb{H}$ then for a.e. $\zeta \in \mathbb S^{2n-1}$ then the zeros of $F_\zeta$ all belong to the 
same ray. In Theorem \ref{theohar}, it is not enough to assume that for a.e.   $\zeta \in \mathbb{S}^{2n-1}$ the zeros of $F_\zeta$ just belong to the same line. For instance consider any polynomial $P:\C\to \C$ with $P(0)=1$ and with zeros on the same line but not on the same ray. Let $\eta \in \C^n\setminus\{0\}$. Then the function 
$F(Z)=P(Z\cdot \eta)$ is such that the zeros of $F_\zeta$ for all $\zeta \in \mathbb S^{2n-1}$ are on the same line but is not of the form of Theorem \ref{theohar}.
\end{example}

\vskip 0,3cm
{\small
\noindent Faruk Abi-Khuzam\\
Department of Mathematics, American University of Beirut, Beirut, Lebanon\\
{\sl E-mail address}: farukakh@aub.edu.lb
\\

\noindent Florian Bertrand\\
Department of Mathematics, American University of Beirut, Beirut, Lebanon\\
{\sl E-mail address}: fb31@aub.edu.lb\\

\noindent Giuseppe Della Sala \\
Department of Mathematics, American University of Beirut, Beirut, Lebanon\\
{\sl E-mail address}: 	gd16@aub.edu.lb\\
}


\begin{thebibliography}{11111} 

\bibitem{ab} F.  Abi-Khuzam, {\it Meromorphic functions with harmonic $*$-function}, Complex Variables Theory Appl. {\bf 12} (1989), 261-265.

\bibitem{ab2} F.  Abi-Khuzam, personal communications.

\bibitem{ba1} A. II Baernstein, {\it Proof of Edrei's spread conjecture}, Proc. London Math. Soc. {\bf 26} (1973), 418-434.

\bibitem{ba2} A. II Baernstein, {\it Integral means, univalent functions and circular symmetrization}, Acta Math. {\bf 133} (1974), 139-169. 

\bibitem{ba-ta} A. II Baernstein, B. A. Taylor, {\it Spherical rearrangements, subharmonic functions, and $*$-functions in $n$-space}, Duke Math. J. {\bf 43} (1976), 245-268. 

\bibitem{ed} A. Edrei {\it Sums of deficiencies of meromorphic functions. II.}, J. Analyse Math. {\bf 19} (1967), 53-74. 

\bibitem{ed-fu1}A. Edrei, W. H. J. Fuchs, {\it Bounds for the number of deficient values of certain classes of meromorphic functions}, 
Proc. London Math. Soc. {\bf 12} (1962), 315-344. 

\bibitem{ed-fu2}A. Edrei, W. H. J. Fuch, {\it Asymptotic behavior of meromorphic functions with extremal spread. I}, 
Ann. Acad. Sci. Fenn. Ser. A I Math. {\bf 2}  (1976), 67-111.

\bibitem{es-sh} M. Ess\'en, D. F. Shea, {\it An extremal problem in function theory}, Lecture notes in Mathematics Vol. 599, Complex Analysis-Kentucky 1976, Springer-Verlag, New York, 1976.
 

\bibitem{ha-ke}W. K. Hayman, P. B. Kennedy, {\it Subharmonic functions. Vol. I}, London Mathematical Society Monographs, No. 9. Academic Press [Harcourt Brace Jovanovich, Publishers], London-New York, 1976, xvii+284 pp.


\bibitem{kr} S. G. Krantz, {\it Function theory of several complex variables}. Reprint of the 1992 edition. AMS Chelsea Publishing, Providence, RI, 2001. xvi+564 pp.
\bibitem{le-gr} P. Lelong, L. Gruman, {\it Entire functions of several complex variables}, Grundlehren der Mathematischen Wissenschaften [Fundamental Principles of Mathematical Sciences], 282. Springer-Verlag, Berlin, 1986. 
xii+270 pp.

\bibitem{ra-sh}N. V. Rao, D. F. Shea, {\it Growth problems for subharmonic functions of finite order in space}, Trans. Amer. Math. Soc. {\bf 230} (1977), 347-370.

\bibitem{ro-we} J. Rossi, A. Weitsman {\it A unified approach to certain questions in value distribution theory}, J. London Math. Soc. {\bf28} (1983), 310-326.



\end{thebibliography}
\end{document}